\documentclass[12pt]{amsart}
\usepackage{amssymb,enumerate}

\usepackage[all]{xy}

\newcommand{\mm}{\mathfrak m}

\newcommand{\Z}{\mathbb{Z}}

\newcommand{\N}{\mathbb{N}}
\newcommand{\Q}{\mathbb{Q}}
\newcommand{\PC}{\mathbb{P}}

\newcommand{\Fc}{\mathcal{F}}

\newcommand{\xb}{{\bf x}}
\newcommand{\yb}{{\bf y}}

\DeclareMathOperator{\pnt}{\raise 0.5mm \hbox{\large\bf.}}

\DeclareMathOperator{\codim}{codim}

\DeclareMathOperator{\depth}{depth}

\DeclareMathOperator{\Img}{Im}

\DeclareMathOperator{\gr}{gr}
\DeclareMathOperator{\chara}{char}
\DeclareMathOperator{\Tor}{Tor}

\DeclareMathOperator{\Ass}{Ass}

\DeclareMathOperator{\lind}{ld}
\DeclareMathOperator{\glind}{glld}
\DeclareMathOperator{\linp}{lin}
\DeclareMathOperator{\Ker}{Ker}
\DeclareMathOperator{\embdim}{emb\,dim}

\DeclareMathOperator{\reg}{reg}
\DeclareMathOperator{\flatdim}{fd}

\DeclareMathOperator{\projdim}{pd}
\DeclareMathOperator{\ini}{in}

\DeclareMathOperator{\rk}{rank}

\DeclareMathOperator{\Sym}{Sym}
\DeclareMathOperator{\HF}{HF}
\DeclareMathOperator{\height}{height}

\def\+#1{\relax\ifmmode\if\noexpand #1\relax \mathop{\kern
    0pt^+{#1}}\nolimits\else \kern 0pt^+\!#1 \fi\else$^*$#1\fi}

\let\phi=\varphi

\newtheorem{thm}{\bf Theorem}[section]
\newtheorem{lem}[thm]{\bf Lemma}
\newtheorem{cor}[thm]{\bf Corollary}
\newtheorem{prop}[thm]{\bf Proposition}

\newtheorem{quest}[thm]{\bf Question}

\theoremstyle{definition}
\newtheorem{defn}[thm]{\bf Definition}

\theoremstyle{plain}
\newtheorem*{thm*}{Theorem}
\newtheorem*{lem*}{Lemma}
\newtheorem*{cor*}{Corollary}
\newtheorem*{claim*}{Claim}
\newtheorem*{defn*}{Definition}

\theoremstyle{remark}
\newtheorem{rem}[thm]{Remark}

\newtheorem{ex}[thm]{Example}

\title[Absolutely Koszul algebras and the Backelin-Roos property]
{Absolutely Koszul algebras and \\ the Backelin-Roos property}

\author[Conca]{Aldo Conca}
\address{Dipartimento di Matematica, Universit\`{a} di Genova, Via Dodecaneso 35, 16146 Genova,
Italy} \email{conca@dima.unige.it}

\author[Iyengar]{Srikanth B. Iyengar}
\address{Department of Mathematics, University of Utah, Salt Lake City, UT 84112, USA}
\email{iyengar@math.utah.edu}

\author[Nguyen]{Hop D. Nguyen}
\address{Dipartimento di Matematica, Universit\`a di Genova, Via Dodecaneso 35, 16146 Genoa, Italy}
\email{ngdhop@gmail.com}

\author[R\"omer]{Tim R\"omer}
\address{Universit\"at Osnabr\"uck, Institut f\"ur Mathematik, 49069 Osnabr\"uck, Germany}
\email{troemer@uos.de}

\thanks{Iyengar was partly supported by National Science Foundation grant DMS-1201889. Nguyen is grateful to the CARIGE foundation for support.}
\begin{document}

\dedicatory{Dedicated to Ng\^o Vi\d{\^e}t Trung on the occasion of his $60^{th}$ birthday.}

\begin{abstract} We study absolutely Koszul algebras, Koszul algebras with the Backelin-Roos property and their behavior under standard algebraic operations. In particular,  we identify some Veronese subrings of  polynomial rings that have the Backelin-Roos property and conjecture that the list is indeed complete. Among other things, we prove that every universally Koszul ring defined by monomials has the Backelin-Roos property. 
\end{abstract}

\maketitle

\section{Introduction}
\label{sect_intro}
Let $k$ be a field and $R$ a  standard graded $k$-algebra.  Recall that $R$ is said to be \emph{Koszul} if $k$, viewed as an $R$-module via the canonical augmentation $R\to k$,  has a linear free resolution.  This is the commutative version of a notion that goes back to Priddy  \cite{P}. For a recent general overview of this class of algebras in the commutative context, and related ones like universally Koszul algebras \cite{Con1} and the strongly Koszul algebras~\cite{HHR}, see \cite{CDR}. 

Koszul algebras are known to behave well from several points of views. For example, the Koszul property can be completely characterized in terms of  finiteness of the Castelnuovo-Mumford regularity, see Avramov, Eisenbud and Peeva \cite{AE}, \cite{AP}.   On the other hand, not all the Koszul algebras are born equal.  According to Roos \cite{Roos} a  Koszul algebra $R$ is \emph{good}  if for each finitely generated graded $R$-module $M$, its Poincar\'{e} series, which is the generating series of the Betti numbers of $M$, is rational and with denominator  depending only on $R$. Roos described  in \cite{Roos}  examples  of  bad (i.e. not good) Koszul algebras.  

The focus of this work is on a class of good Koszul algebras called absolutely Koszul algebras. It was introduced in \cite{IyR} through the notion of linearity defect as follows.  Consider the minimal graded free resolution $F$ of a finitely generated $R$-module $M$.  The multiplication with the maximal homogeneous ideal $\mm$  of $R$ induces on $F$ a suitable filtration (see Section \ref{sect_general} for details) whose associated graded complex  is denoted by $\linp^R F$. The \emph{linearity defect} of $M$ is  defined as
\[
\lind_R M=\sup\{i: H_i(\linp^R F)\neq 0\}.
\]
This homological invariant has been introduced by Herzog and Iyengar \cite{HIy} building on work of  Eisenbud, Fl{\o}ystad, and Schreyer~\cite{EFS} on resolution of modules over exterior algebras. The ring $R$ is \emph{absolutely Koszul} if $\lind_RM$ is finite for each finitely generated graded $R$-module $M$. 

One of the things that has emerged from the work reported here is that the absolutely Koszul property for $R$, which concerns the structure of all $R$-modules, is closely related to one concerning the algebra structure of $R$. Namely whether there exists a surjective morphism $\varphi\colon Q\to R$ of standard graded $k$-algebras such that the ring $Q$ is a complete intersection and the map $\varphi$ is Golod. When such a map exists we say that $R$ has the \emph{Backelin-Roos property}; see Section~\ref{BRprop} for details. For a Koszul algebra $R$ one has the following implications: 
 \[
\begin{array}{c}
R \mbox{ has the Backelin-Roos property} \\
\Downarrow \\ 
R \mbox{ is absolutely Koszul} \\
\Downarrow \\ 
 R \mbox{ is good in the sense of Roos.}  
 \end{array} 
 \] 
These implications are proved in \cite[\S1.8]{HIy} and \cite[Theorem 5.9]{HIy}. We are not aware of any examples that show the implications are strict. In other words, every absolutely Koszul algebra that we know of has the Backelin-Roos property (at least after a field extension), and the ones that are not absolutely Koszul rings are bad Koszul algebras in the sense of Roos.  However if one drops the assumption that $R$ is Koszul, there are examples of algebras $R$ that are good in the sense of Roos without having  the Backelin-Roos property, see \cite[pg.30-31]{A} and \cite[Sect.6]{KP}. 

The first goal of the paper is to track the absolutely Koszul and the Backelin-Roos properties under standard algebra operations. This is the content of Sections \ref{sect_general} and \ref{BRprop}.  The focus of Section \ref{sect_monomial} is on Koszul algebras with monomial relations, and one of the main results there is that any universal Koszul algebra with monomial relations has the Backelin-Roos property.  

Section \ref{VeroneseSegre} is devoted to identifying Veronese and Segre algebras that have the Backelin-Roos property. When $k$ is algebraically closed and of characteristic zero, we  show that a  Cohen-Macaulay domain $R$ with  $e(R)\leq 2\codim(R)$ has  the Backelin-Roos property  and, under some   additional assumptions, the same is true also for Gorenstein domains  with $e(R)=2\codim(R)+2$.   These statements are proved in Theorems \ref{thm:geoint}  using a reduction to the Artinian case,  exploiting  a result  of Harris on the general hyperplane sections of irreducible projective curves  stated in  \cite[page 109]{ACGH}, and results of Conca, Rossi and Valla \cite{CRV} on the Koszul property of points in projective space and of \c{S}ega and Henriques on the  Backelin-Roos property of   Gorenstein Artinian algebras with socle in degree $3$. From this  one deduces that the $c$-th Veronese $S^{(c)}$ subalgebra of the polynomial ring $S$ in $n$ variables over an algebraically closed field of characteristic $0$ has the Backelin-Roos property  in the following cases: 
\begin{enumerate}[\quad\rm(1)]
\item $n\le 3$ and any $c$; 
\item $n=4$ and $c\le 4$;
\item $n\le 6$ and $c\le 2$;  
\end{enumerate} 
See Corollary \ref{cor:Veronese} and Corollary~\ref{cor:anyk}. 

On the other hand, when $(n,c)$ equals $(4,5)$, or $(5,3)$, or $(7,2)$,  the algebra  $S^{(c)}$ does not have the Backelin-Roos property;
see Lemma~\ref{lem:obstructedVeronese}. Indeed the $h$-polynomial $h_R(z)$ of an algebra $R$ with the Backelin-Roos property must satisfy  the inequality $h_R(-1)\leq 0$; see Corollary~\ref{cor:HilbObstruction}, and this condition is violated for the values of $(n,c)$ listed above.  A more stringent numerical obstruction is described  in Proposition~\ref{prop:HilbObstruction}. In fact, computational evidence suggest that all the Veronese algebras $S^{(c)}$ with $(n,c)$ different from the values in (1),(2), and (3) above violate the inequality. A similar situation occurs for Segre products of polynomial rings. 

We are thus faced with the following intriguing question:

\medskip

 \emph{Are all Veronese subrings  of polynomial rings absolutely Koszul? What about Segre products}?

\medskip
 
Finally in Section \ref{sect_glind}  we discuss the global linearity defect $\glind R$ of $R$, defined as the supremum of all the linearity defects of modules over $R$. This number invariant is rarely finite; see \cite[Theorem 6.2]{HIy}. Our goal in this section is to find bounds on the global linearity defect.

The first author would like to thank Volkmar Welker for useful discussions concerning the Hilbert series  of Veronese algebras. 

\section{Absolutely Koszul algebras and changes of rings}
\label{sect_general}
In this section we establish some results, and recall earlier ones, that track the absolutely Koszul property under change of rings.  Throughout, $k$ will  be a field and $R$ a standard graded $k$-algebra.   Let $M$ be a finitely generated graded $R$-module and
\[
F: \cdots \to F_i\to \cdots \to F_1\to F_0\to 0
\]
its minimal graded free resolution. For each $i\ge 0$, consider the subcomplex $\Fc^i F$ of $F$ given by 
\[
\Fc^i F: \cdots \to F_{i+1} \to F_i \to \mm F_{i-1} \to \cdots \to \mm^{i-1}F_1\to \mm^iF_0 \to 0.
\]
where $\mm=R_{\geqslant 1}$ is the homogeneous maximal ideal of $R$. The associated graded complex of the descending filtration $\{   \Fc^{i}F \}_{i\in \N}$   is denoted  by $\linp^R F$, and called the \emph{linear part of $F$}. Observe that $(\linp^R F)_i=\gr_{\mm}(F_i)(-i)$ and that the matrices of differential of $\linp^R F$ are obtained from those of $F$ by replacing all entries of degree $\ge 2$ by zero. 

The \emph{linearity defect} of $M$ is  defined as
\[
\lind_R M=\sup\{i: H_i(\linp^R F)\neq 0\}.
\]
By definition, $R$ is absolutely Koszul if $\lind_R M$ is finite for every $M$.  R\"omer proved that when $R$ is Koszul, a module has linearity defect zero if and only if it is componentwise linear in the sense of Herzog and Hibi \cite{HH}; see, e.g., \cite[Theorem 5.6]{IyR}.

To prepare the reader to better appreciate the results below, we recall an example that shows that, unlike for Koszul algebras,  the class of absolutely Koszul algebras is not closed under tensor product over $k$.   

\begin{ex}
\label{bad-Koszul}
Let $k$ be a field and set $R=k[x,y]/(x,y)^2$, where both $x$ and $y$ are of degree one. This ring is absolutely Koszul: evidently it is Koszul, so $\lind_{R}k=0$, and the first syzygy module of any $R$-module is direct sum of shifted copies of $k$. In fact, it is not hard to verify directly from definitions that the linearity defect of any $R$-module is zero.

However, $R\otimes_{k}R$ is Koszul but not absolutely Koszul. This is because  $R$ is a bad Koszul algebra, in the sense of Roos~\cite[Theorem 2.4b(A)]{Roos}.

Another such example is the ring $R=k[x,y]/(x^{2},xy)$: it is absolutely Koszul but $R\otimes_{k}R$ is not; see \cite[Theorem 2.4b(C)]{Roos}.
\end{ex}

\subsection*{Change of rings}
We present a few results concerning the ascent and descent of the absolutely Koszul property  along a morphism $\varphi\colon R\to S$ of standard graded $k$-algebras. Following \cite{IyR}, we call the number
\[
\glind R=\sup\{\lind_R M: M ~ \text{a finitely generated graded $R$-module}\}
\]
the \emph{global linearity defect} of the $k$-algebra $R$. 

\begin{lem}
\label{lem:field-extension}
Let $R$ be a standard graded $k$-algebra and $k\subseteq l$ an algebraic extension of fields. Set $S=l\otimes_{k}R$ and let $\varphi\colon R\to S$ be the canonical inclusion. The following statements hold.
\begin{enumerate}[\quad\rm(1)]
\item $\lind_{S}(l\otimes_{k}M)=\lind_{R}M$ for any finitely generated graded $R$-module $M$.
\item $\lind_{S} N = \lind_{R}N$ if $k\subseteq l$ is a finite extension and $N$ is a finitely generated graded $S$-module. 
\item $\glind S= \glind R$ holds.
\item $R$ is absolutely Koszul if and only if $S$ is. 
\end{enumerate}
\end{lem}

\begin{proof}
The  morphism $\varphi$ is flat and $R_{\geqslant 1}S = S_{\geqslant 1}$; this remark will be used repeatedly in the proof.

(1) It is not hard to verify that if $F$ is the minimal graded free resolution of a finitely generated graded $R$-module $M$, then the complex of $S$-modules $l\otimes_{k}F$ is a minimal graded free resolution of the $S$-module $l\otimes_{k}M$, and that there is an isomorphism of complexes
\[
\linp^{S}(l\otimes_{k}F) \cong l\otimes_{k}\linp^{R}F\,.
\]
This explains (1).

\medskip

(2) Let $G$ be a minimal graded free resolution of $N$ as an $S$-module. The extension $k\subseteq l$ is finite, so $S$ is a finitely generated free $R$-module. It follows that $G$ is also the minimal graded free resolution of $N$ viewed as an $R$-module via $\varphi$, and that there is an isomorphism of complexes
\[
\linp^{R}G \cong \linp^{S}G\,.
\]
This implies the desired equality.

\medskip

(3) and (4) When $S$ is absolutely Koszul, it is immediate from (1) that $R$ is absolutely Koszul as well, and that there is an inequality $\glind R\leq \glind S$.  For the converse: let $N$ be a finitely generated graded $S$-module. Observe that $N$ can be realised as $l\otimes_{l'}N'$, where $l'$ is a \emph{finite} extension of $k$ and $N'$ is a finitely generated graded  module over $l'\otimes_{k}R$. Then from (1) and (2) one gets equalities
\[
\lind_{S}N = \lind_{(l'\otimes_{k}R)}(N') = \lind_{R} N'.
\]
This implies that if $R$ is absolutely Koszul, so is $S$ and also yields $\glind S\leq \glind R$. 
\end{proof}

In preparation for proof of the next result we recall that the  $(i,j)$-graded Betti number of finitely generated graded $M$ over a standard graded $k$-algebra is the integer
\[
\beta^{R}_{i,j}(M)=\dim_k \Tor^R_i(k,M)_j\,.
\]
The graded Poincar\'{e} series  of $M$ is the formal power series
\[
P^R_M(s,z)=\sum_{i\in \Z}\sum_{j\in \Z}\beta^{R}_{i,j}(M)s^{j}z^i \in \Z[s][|z|];
\]
it has non-negative coefficients.

\begin{prop}
\label{prop:retracts}
Let $\varphi\colon R\to S$ be a morphism of standard graded $k$-algebras. If there exists a morphism $\sigma\colon S\to R$ of $k$-algebras such that $\varphi\sigma\colon S\to S$ is the identity, then the following statements hold.
\begin{enumerate}[\quad\rm(1)]
\item
If $\lind_{R}S=0$, then $\lind_S N=\lind_RN$  for every finitely generated graded $S$-module $N$.
\item
When $R$ is Koszul, so is $S$ and $\lind_{R}S=0$.
\item
When $R$ is absolutely Koszul so is $S$ and there is an inequality $\glind S\leq \glind R$.
\end{enumerate}
 \end{prop}

\begin{proof}

(1)  Let $F$ be the minimal graded $R$-free resolution of $S$ and $G$  the minimal graded $S$-free resolution of $N$. The complex of $R$-modules $F\otimes_{S}G$, where $S$ acts on $F$ by restriction of scalars along $\sigma$ and $R$ acts on the tensor product via its action on $F$, is then the  minimal graded $R$-free resolution of $N$: Indeed, it is clearly minimal, and as $G$ is a complex of free $S$-modules, the quasi-isomorphism $F\xrightarrow{\simeq} S$ induces the first quasi-isomorphism below
\[
 F\otimes_{S}G\xrightarrow{\simeq} S\otimes_{S}G\simeq N\,.\tag{$\ast$}
 \]
See \cite[Theorem 1]{Her} for details. Arguing as in the proof of \cite[Lemma 2.7]{IyR} one gets an isomorphism of complexes of $R$-modules:
\[
\linp^R (F\otimes_S G)\cong \linp^R F \otimes_S \linp^S G\,.
\]
Since $\linp^R F \simeq S$, as $\lind_RS=0$, and $\linp^S G$ is a complex of free $S$-modules, the canonical map
\[
\linp^R F \otimes_S \linp^S G \xrightarrow{\quad \simeq\quad} S\otimes_{S} \linp^SG \cong \linp^S G
\]
is a quasi-isomorphism. This gives the desired equality.

\medskip

(2) The isomorphism ($\ast$) above, with $N=k$, implies an equality of formal power series
\[
P^{R}_{k}(s,z) = P^{R}_{S}(s,z)P^{S}_{k}(s,z)\,.
\]
When $R$ is Koszul, $\beta^{R}_{i,j}(k)=0$ for all $i\ne j$. Recall that one always has inequalities
\[
\beta^{R}_{i,j}(S)=0=\beta^{S}_{i,j}(k)\quad \text{for $i>j$}
\]
so it follows from the equality above and the condition on $\beta^{R}_{i,j}(k)$ that  $\beta^{R}_{i,j}(S)=0=\beta^{S}_{i,j}(k)$ also for $i<j$.  Thus $S$ has a linear resolution over $R$, so that $\lind_{R}S=0$, and also $S$ is Koszul.

\medskip

(3) is a direct consequence of (1) and (2).
\end{proof}

Recall that  a \emph{flat resolution} of a (not necessarily finitely generated) $R$-module $M$ is a complex $F$ of flat $R$-modules with $F_{i}=0$ for $i<0$ equipped with a quasi-isomorphism $F\simeq M$. The \emph{flat dimension} of $M$ is 
\[
\flatdim_{R}M:=\inf\left\{n\ge 0 \left|
\begin{gathered}
\text{there exists a flat resolution $F\simeq M$}\\
\text{ such that $F_{i}=0$ for $i\not\in[0,n]$}
\end{gathered}
\right.
\right\}.
\]
This equals the projective dimension of $M$ when $M$ is finitely generated. The result below extends \cite[Theorem 2.11]{IyR} where it is assumed also that $S$ is a finitely generated $R$-module.

\begin{thm}
\label{thm:glind-flatdim}
Let $R\to S$ be a morphism of standard graded $k$-algebras such that $\flatdim_R S$ is finite.  If $S$ is absolutely Koszul, then so is $R$, and there is an inequality
\[
\glind R\le \glind S+\flatdim_R S + \dim (S/\mm S)\,,
\]
where $\mm$ is the homogeneous maximal ideal of $R$.
\end{thm}

\begin{proof}
We reduce to the case where the $R$-module $S$ is finitely generated, as follows.

Enlarging the field $k$ if necessary, we may assume it is infinite; see Lemma~\ref{lem:field-extension}. Then, by Noether normalisation, there exist elements $s_{1},\dots,s_{c}$ in $S_{1}$ that form a system of parameters for the ring $S/\mm S$. Set $R'=R[y_1,\ldots,y_c]$, where $\yb:=\{y_{1},\dots,y_{c}\}$ are indeterminates in degree one, let $\dot{\varphi}\colon R\to R'$ be the canonical inclusion, and $\varphi'\colon R'\to S$ the morphism of $R$-algebras defined by the assignment $y_{i}\mapsto s_{i}$ for each $i$. These give a factorisation
\[
R\xrightarrow{\quad\dot\varphi\quad} R' \xrightarrow{\quad\varphi'\quad} S
\]
of $\varphi$. Note that $\mm' = \mm + (y_{1},\dots,y_{c})$ is the homogeneous maximal ideal of $R'$ and that  $\rk_{k}(S/\mm' S)$ is finite, by construction. Therefore, $S$ is finitely generated when viewed as an $R'$-module via $\varphi'$; this is by Nakayama's Lemma for graded modules.

Consider the morphism $R'\to R'/(\yb)\cong R$. The Koszul complex on $\yb$ is a minimal graded free resolution of $R$ as a module over $R'$. Given this it is clear that $\lind_{R'}R=0$. Moreover, composing $\dot\varphi$ with $R'\to R$ gives the identify on $R$, so Proposition~\ref{prop:retracts} applies and yields that when $R'$ is absolutely Koszul so is $R$, and that there is an inequality
\[
\glind R\leq \glind R'\,.
\]
We claim that there is an inequality 
\[
\flatdim_{R'} S \le \flatdim_R S+ c\,.
\]
Indeed, this holds because $\dot\varphi$ is a flat morphism and $R'/\mm R'$ is a $c$-dimensional regular ring; see \cite[(3.3)]{AFH} for an argument in the case of local homomorphisms; it carries over to the graded case.

In summary, replacing $R$ by $R'$ we may assume that $S$ is finitely generated as an $R$-module, and of finite projective dimension.
In this case \cite[Theorem 2.11]{IyR} applies---noting once again that whilst \emph{op.cit.} deals with local rings the argument carries over to the graded case---and gives the desired statement.
\end{proof}

\begin{cor}
\label{cor:poly}
Let $R$ be a standard graded $k$-algebra and $x$ an indeterminate with $\deg x=1$. Then $R$ is absolutely Koszul if and only if so is $R[x]$.
\end{cor}

\begin{proof}
Indeed, both implications can be deduced from Theorem~\ref{thm:glind-flatdim}, by applying to the canonical inclusion $R\to R[x]$ and the canonical surjection $R[x]\to R$.
\end{proof}

\subsection*{Fibre products}
Let  $S$ and $T$ be standard graded $k$-algebras, and $\varepsilon_{S}\colon S\to k$ and $\varepsilon_{T}\colon T\to k$ the canonical augmentations. The \emph{fibre product} of $S$ and $T$ is the ring 
\[
S\times_{k}T : = \{(s,t)\in S\times T\mid \varepsilon_{S}(s) = \varepsilon_{T}(t)\}.
\]
Thus, if $S=k[x_1,\ldots,x_m]/I$ and $T=k[y_1,\ldots,y_n]/J$, then
\[
S\times_k T \cong \frac{k[x_1,\ldots,x_m,y_1,\ldots,y_n]}{I+J+(x_iy_j:1\le i\le m, 1\le j\le n)}.
\]
Observe that the fibre product is itself a standard graded $k$-algebra.

\begin{thm}
\label{thm:fibre-product}
The fibre product $S\times_k T$ is absolutely Koszul if and only if  $S$ and $T$ are.
Moreover, when $\glind S$ and $\glind T$ are finite so is $\glind (S\times_k T)$.
\end{thm}

\begin{proof}
It is well-known that $S\times_k T$ is Koszul if and only if $S$ and $T$ are Koszul; see \cite{BF}. Hence we assume from the beginning that $S$ and $T$, and hence also $S\times_{k}T$, are Koszul. Set $R=S\times_k T$, and note that the augmentations of $S$ and $T$ induce morphisms 
\[
\varepsilon_{S}\times_{k}T\colon R\to T\quad\text{and}\quad S\times_{k}\varepsilon_{T}\colon R\to S
\]
of graded $k$-algebras. Observe also that the composition of the inclusion $T\to R$  with $\varepsilon_{S}\times_{k}T$ is the identity on $T$, whilst that of $S\to R$ with $S\times_{k}\varepsilon_{T}$ is the identity on $S$. Proposition~\ref{prop:retracts}(3) thus applies and yields that when $R$ is absolutely Koszul, so are $S$ and $T$. The same result also implies that when $S$ and $T$ are absolutely Koszul, any $R$-module that is obtained from $S$ by restriction along $S\times_{k}\varepsilon_{T}$, or from $T$  by restriction along  $\varepsilon_{S}\times_{k}T$,  has finite linearity defect. It remains to note that the second syzygy of any finitely generated graded $R$-module is of the form $N\oplus P$ where $N$ is an finitely generated graded $S$-module and $P$ is a finitely generated graded $T$-module; see~\cite[Bemerkung 3]{DK}.
\end{proof}

\section{The Backelin-Roos property}
\label{BRprop} 

Let $\varphi\colon Q\to R$ be a surjection of standard graded $k$-algebras such that $\Ker \varphi \subseteq (Q_{\ge 1})^2$. By a result of Serre, see for example \cite[Prop.3.3.2]{Avr}, there is a coefficientwise inequality of formal power series:
\[
P^R_k(s,z) \preccurlyeq \frac{P^Q_k(s,z)}{1-z(P^Q_R(s,z)-1)}.
\]
When there is equality, $\varphi$ is said to be \emph{Golod}.  The ring $R$ is said to be Golod if some (equivalently, any) minimal presentation $k[x_{1},\dots,x_{n}]\twoheadrightarrow R$, meaning that $x_{1},\dots,x_{n}$ are indeterminates over $k$ and $n=\embdim R$, is a Golod homomorphism.

For what follows, it is helpful to have the following notation: For any finite generated graded $Q$-module $M$ and integer $i\ge 0$, set
\[
t^{Q}_{i}(M) = \max\{j\in\Z \mid \Tor^{Q}_{i}(k,M)_{j}\ne 0\}.
\]
Here is a simple test for detecting Golod maps when the source and target are Koszul; this was proved by Herzog and Iyengar~ \cite[Proposition 5.8]{HIy} in the local case, but the argument applies also in our context.

\begin{rem}
\label{rem:Golod2lin}
The following statements are equivalent:
\begin{enumerate}[\quad\rm(i)]
\item $R$ is a Koszul algebra and $\varphi$ is a Golod morphism;
\item $Q$ is a Koszul algebra and $t^{Q}_{i}(R)\leq i+1$ for each $i\ge 0$.
\end{enumerate}
In (ii), the condition on the $t^{Q}_{i}(R)$ can be stated as: $\Ker(\varphi)$ has a 2-linear $Q$-free resolution. A special case of this equivalence  is  worth noting: The ring $R$ is Koszul and Golod if, and only, if the kernel of any minimal presentation $k[x_{1},\dots,x_{n}]\twoheadrightarrow R$ has a 2-linear free resolution.
\end{rem}

The elementary observation below will prove useful.

\begin{rem}
\label{rem:tis}
Given a complex of graded $Q$-modules
\[
C\colon \cdots \to C_i \to C_{i-1} \to \cdots \to C_0\to 0\,.
\]
A simple computation, based on breaking $C$ into short exact sequences, yields that for each integer $i\geq 0$, there is an equality
\[
t_i^Q(H_0(C))\leq \max\{a_i,b_i\}
\]
where the integers $a_{i}$ and $b_{i}$ are defined as follows:
\begin{align*}
a_i & =\max\{ t_j^Q(C_{i-j}) : j=0,\dots, i\} \\
b_i &=\max\{ t_j^Q(H_{i-j-1}(C)) : j=0,\dots,i-2\}.
\end{align*}
\end{rem}

As before, let $R$ be a standard graded $k$-algebra. 

\begin{defn} 
We say that $R$ has the \emph{Backelin-Roos property} if there exists a surjective Golod morphism $Q\to R$, with $Q$ a standard graded complete intersection ring. 
\end{defn} 

For example, when $R$ is either a complete intersection or is Golod, it has the Backelin-Roos property. Observe that when $R$ is Koszul, any $Q$ as above also has to be Koszul, by Remark~\ref{rem:Golod2lin}, and hence defined by quadrics. 

The next result is a version of \cite[Theorem 5.9]{HIy} for graded rings; again the proof carries over. It is essentially the only method we know for detecting absolutely Koszul rings.

\begin{prop}
\label{prop:br=ak}
Let $R$ be a standard graded Koszul $k$-algebra. If for some algebraic extension $l$ of $k$, the ring $l\otimes_{k}R$ has the Backelin-Roos property, then $R$ is absolutely Koszul.
\end{prop}

\begin{proof}
By \cite[Theorem 5.9]{HIy}, the ring $l\otimes_{k}R$ is absolutely Koszul; now apply Lemma~\ref{lem:field-extension}(4).
\end{proof}

In the remainder of this section we establish some techniques for detecting when a ring has the Backelin-Roos property. Since the focus of this work is on Koszul algebras, we consider only this class of algebras, though some of the results (for example, Lemma~\ref{lem:br-dim} below) hold without this additional assumption. The Koszul property makes it easier to test when a morphism is Golod; see Remark~\ref{rem:Golod2lin}.

\begin{lem}
\label{lem:br-dim}
Assume that $k$ is an infinite field. Let $R$  be standard graded Koszul $k$-algebra. If $R$ has the Backelin-Roos property then there exists a Golod morphism $Q\to R$ where $Q$ is a complete intersection whose dimension and  codimension equal that of $R$. 
\end{lem} 

\begin{proof}
By hypothesis, there exists a surjective Golod morphism $P \to R$ with $P$ a standard graded complete intersection ring. The Golod property implies that the kernel of $P\to R$ is contained in $P_{\geqslant 2}$. Hence the embedding dimensions of $P$ and $R$ coincide. It thus suffices to prove that one can choose $P$ such that its dimension equals that of $R$.

Suppose $\dim P>\dim R$. Let  $I$ and $J$ be the defining ideals of $P$ and $R$ as quotients of the same polynomial ring. We have $I\subset J$ and $\height I<\height J$, and $I$ is generated by a regular sequence of quadrics. As $k$ is infinite, by prime avoidance there exists a quadric $f$ in $J$ that is regular over $P$. Set $Q=P/(f)$. Then $Q$ is a complete intersection of quadrics and hence Koszul. We prove that $Q\to R$ is Golod by verifying that $t_i^{Q}(R)\leq i+1$ for every $i$; see Remark~\ref{rem:Golod2lin}.

Let $F$ be the minimal graded free resolution of $R$ as a $P$-module and set $G=F\otimes_P Q$. Then one has $H(G)=\Tor^P(R, Q)$.  Since $f$ is regular over $P$ and $fR=0$ it follows that 
\[
H_i(G)=\left\{
\begin{array}{ll}
R & \mbox{ for } i=0\\ 
R(-2)  & \mbox{ for } i=1\\
0 & \mbox{ for } i>1
\end{array}
\right.
\]
Applying Remark~\ref{rem:tis} to the  complex of $Q$-modules $G$ and taking into consideration that  $G_i$ is $Q$-free and generated in degree $0$ for $i=0$ and in degree $i+1$ when $i>0$, one obtains: 
\[ 
t_i^{Q}(R)  \leq \max\{ i+1, t_{i-2}^{Q}(R)+2\}\,.
\]
Since $t_0^{Q}(R)=0$ and $t_1^{Q}(R)=2$ by construction, it follows by induction that $t_i^{Q}(R)\leq i+1$. 
\end{proof}

The following example shows that when $R$ has the Backelin-Roos property not all the maximal regular sequences contained in its defining ideal induce a Golod map. 

\begin{ex} Set $P=k[a,b,c,d]$, $I=(a^2, b^2, ad,  ac, bd )$ and  $R=P/I$.  Note that $(ad,  ac, bd)$ has a $2$-linear resolution over $P$. Hence, by virtue of Theorem~\ref{thm:ciplus2linear} proved in the next section,  the algebra  $R$ has the Backelin-Roos property.  On the other hand $J=(ac, bd)$ is a  maximal  regular sequence in $I$ and the map $Q:=P/J\to R$ is not Golod since $\beta_{2,4}^Q(R)=1$. 
\end{ex} 

\begin{lem}
\label{lem:br-ascent}
Let $R$ be a standard graded  $k$-algebra and $x_{1},\dots,x_{n}$ a regular sequence  in $R_{1}$. If $R/(x_{1},\dots,x_{n})$ is a Koszul algebra with the Backelin-Roos property, then so is $R$.
\end{lem}

\begin{proof}
It suffices to verify the assertion for $n=1$. So assume $x\in R_{1}$ is a non-zerodivisor and that there exists a surjective morphism of standard graded $k$-algebras 
\[
\varphi \colon B\to R/(x)
\]
such that $B$ is a complete intersection, necessarily of quadrics. Let $x_1,\dots, x_e$ be a $k$-basis of $R_1$  with  $x=x_e$. We may suppose  $B=P/J$  where $P=k[y_1,\dots,y_{e-1}]$ is the polynomial ring, $J$ is generated by a regular sequence in $P$, and that the residue class of $y_{i}$ is the residue class of $x_{i}$, for $i=1,\dots, e-1$.

Say $J=(f_1,\dots, f_c)$ and consider the morphism of $k$-algebras $\pi\colon P[y_e]\to R$ that sends $y_i$ to $x_i$ for $i=1,\dots,e$. Since $f_i(x_1,\dots,x_{e-1})=0 \mod(x_e)$, there exist homogeneous elements  $g_i$ in $P$ such that $f_i(x_1,\dots,x_{e-1})=x_eg_i(x_1,\dots,x_e)$, that is to say, $f_i-y_eg_i$ are in $\Ker\pi $. Set 
\[
A=P[y_e]/J_1\quad\text{where}\quad J_1=(f_i-y_eg_i\mid  i=1,\dots, c)\,.
\]
Since $J_1\subset \Ker \pi$ we have an induced morphism of $k$-algebras $\vartheta\colon A\to R $ sending $y_i \mod J_1$ to $x_i$ for $i=1,\dots,e$.  Now, since 
\[
J_1+(y_e)=J+(y_e)=(f_1,\dots,f_c,y_e)
\]
is generated by regular sequence, it follows that  $f_1-y_eg_1, \dots, f_c-y_eg_c$ is a regular sequence of quadrics in $P[y_{e}]$, and that $y_e$ is regular on $A$. Since $A/(y_e)=B$ there is an isomorphism of graded $k$-vector spaces
\[
\Tor_i^A(k,R)\cong \Tor_i^B(k,R/(x)).
\]
Hence $t^{A}_{i}(R) =t^{B}_{i}(R/(x))$. Noting that $A$ and $B$ are Koszul algebras, it follows from Remark~\ref{rem:Golod2lin} that since $B\to R/(x)$ is Golod, so is $A\to R$, and that $R$ is Koszul.
\end{proof}

\begin{thm}
\label{thm:br-fibre-product}
Let $S$ and $T$ be standard graded Koszul $k$-algebras. If  $S$ and $T$ have the Backelin-Roos property then so does the fibre product $S\times_k T$.
\end{thm}

\begin{proof}
Assume that $S$ and $T$ have the Backelin-Roos property, so that there are  Golod morphisms $S_{1}\to S$ and $T_{1}\to T$ where $S_{1}$ and where $T_1$ are complete intersections, necessarily of quadrics, for $S$ and $T$ are Koszul. Set $Q=S_1\otimes_k T_1$ and $R=S\times_{k}T$. Observe that $Q$ is a complete intersection of quadrics, and in particular Koszul. Consider the composite morphism 
 \[
Q\to S \otimes_k T \to R\,.
\]
We claim that this map is Golod, and verify it by proving that $t^Q_i(R)\leq i+1$ for every $i$; see Remark~\ref{rem:Golod2lin}. To that end, consider the exact sequence of $Q$-modules
\[ 
0\to R \to S \oplus T \to k\to 0\,.
\]
From this it follows that for $i\in\N$ there are inequalities 
\[
t^Q_i(R)\leq \max\{ t^{Q}_i(S), t^{Q}_i(T), t^{Q}_{i+1}(k) \}\,.
\]
It suffices to verify that each term on the right is bounded by $i+1$.
 
Now, since $Q$ is Koszul, one has $t^Q_{i+1}(k) =i+1$ for $i\ge 0$.
 
As a $Q$-module, one can identify $S$ with $S\otimes_k k$, so its minimal graded free resolution is obtained by tensoring,  over $k$, the minimal graded free resolution of $S$ as a $S_1$-module with the minimal graded free resolution of $k$ as a $T_1$-module. It follows that 
\[
t_i^{Q}(S)\leq \max\{ t^{S_1}_j(S)+ t^{T_1}_{i-j} (k) \mid j=0,\dots i\}.
\]
Since $t^{S_1}_j(S)\leq j+1$ and $t^{T_{1}}_{i-j} (k)=i-j$, by assumption, one obtains $t_i^Q(S)\leq i+1$. 
 
By symmetry $t_i^Q(T)\leq i+1$. 
 \end{proof}
 
 \begin{lem}
\label{lem:bcr-regseq}
Let $R$ be a standard graded Koszul $k$-algebra that has  the Backelin-Roos property.  Let $Q$ be a complete intersection of quadrics and $Q\to R$ a Golod map. 
\begin{itemize}
\item[(1)]  If a quadric $f\in Q$ is $Q$-regular and $R$-regular then $Q/fQ\to R/fR$ is a  Golod map  and hence $R/fR$ has the Backelin-Roos property. 
\item[(2)] $R[x]$ and $R[x]/(x^2)$ are Koszul and have the Backelin-Roos property.
\end{itemize}
\end{lem} 

\begin{proof}
(1) Set $A=R/fR$ and $B=Q/fQ$. Let $F$ be the minimal free resolution of $R$ as a $Q$-module. Since, by assumption, $f$ is a regular over $R$ and $Q$ we have that $F\otimes B$ is the minimal free resolution of $A$ as a $B$-module. Hence $t_i^B(A)=t_i^Q(R)=i+1$ for every $i>0$. Since $A$ is Koszul and $B$ is a complete intersection of quadrics, it follows that the map $B\to A$ is Golod and hence $A$ has the Backelin-Roos property. 

(2) The extension $Q[x]\to R[x]$ of  $Q\to R$ is clearly Golod. Hence $R[x]$ has the Backelin-Roos property. Then (1) implies that    $R[x]/(x^2)$  has the Backelin-Roos property as well.
\end{proof}

\begin{rem} 
In the table below we collect what we know concerning the transfer of the absolutely Koszul and of the Backelin-Roos properties with respect to standard algebraic operations. Let $R$ be a Koszul algebra, and let $\ell \in R_{1}$ and  $q\in R_{2}$ be non-zerodivisors.

\begin{table}[htdp]
\begin{center}
\begin{tabular}{|l|c|c|}
\hline
     &  Absolutely Koszul   &  Backelin-Roos \\
\hline\hline 
Fiber product  & yes (\ref{thm:fibre-product}) & yes (\ref{thm:br-fibre-product}) \\
\hline
Ascend from $R/\ell R$ to $R$  & yes (\ref{thm:glind-flatdim}) 
 & yes (\ref{lem:br-ascent})
  \\
\hline
Descend from $R$ to $R/\ell R$ & ?? & ??\\
\hline
Ascend from $R/qR$ to $R$  & yes (\ref{thm:glind-flatdim})  & ??  \\
\hline
Descend from $R$ to $R/qR$ & ?? & ??\\
\hline
Preserved by algebra retracts & yes (\ref{prop:retracts})  & ?? \\ 
\hline
Preserved by the extension $R\to R[x]/(x^2)$ & ?? & yes (\ref{lem:bcr-regseq})  \\
\hline
Preserved by the extension $R\to R[x]$ & yes (\ref{cor:poly})   & yes (\ref{lem:bcr-regseq})  \\
\hline
\end{tabular}
\end{center}
\end{table}%
 \end{rem}

Recall that $S$ is said to an algebra retract of $R$  if there are morphisms of rings $S\xrightarrow{\sigma}R\xrightarrow{\varphi}S$ such that the composition $\varphi\sigma$ is the identify on $S$.  Perhaps the most interesting open question related to the table above is the following. 
 
 \begin{quest} 
 \label{que:retract} 
 Let $R$ be a  standard graded $k$-algebra that is Koszul and has the Backelin-Roos  property. If $S$ is an algebra retract of $R$, does it also have the Backelin-Roos property? 
 \end{quest} 
 
The following is a numerical  obstruction to the Backelin-Roos property in terms of Hilbert functions. It is  reminiscent of the one for the Koszul property itself, namely,  the formal power series $H_{R}(-t)^{-1}$ should have non-negative coefficients.

\begin{prop}
\label{prop:HilbObstruction}
Let $R$ be a standard graded Koszul $k$-algebra with the Backelin-Roos property. The following formal power series, where $c$ is the codimension of $R$ and  $h_R(z)$ its $h$-polynomial, has non-negative coefficients:
\[
1-\frac{h_R(-z)}{(1-z)^c}.
\]
\end{prop}

\begin{proof}
Enlarging the field if necessary, we can assume that $k$ is infinite. Let $Q\to R$ be a Golod morphism where $Q$ is a standard graded complete intersection, with dimension and codimension equal to that of $R$; see Lemma~\ref{lem:br-dim}. Then, by Remark~\ref{rem:Golod2lin}, the minimal graded free resolution of $R$ over $Q$ has the form
\[
\cdots \to Q(-i-1)^{\beta_i}\to Q(-i)^{\beta_{i-1}} \to \cdots \to Q(-2)^{\beta_1}\to Q \to 0\,.
\]
Therefore the Hilbert series of $R$ is
\[
H_R(z)=H_Q(z)(1-\beta_1z^2+\beta_2z^3-\cdots)=\frac{(1-z^2)^c}{(1-z)^m}(1-\beta_1z^2+\beta_2z^3-\cdots)\,,
\]
where $m=\embdim R$. On the other hand  
\[
H_R(z)=\frac{h_R(z)}{(1-z)^{m-c}}.
\]
Equating the two expressions for $H_R(z)$ yields an equality
\[
1-\frac{h_R(z)}{(1+z)^c}=\beta_1z^2-\beta_2z^3+\cdots\,.
\]
Replacing $z$ by $-z$ yields the desired statement.
\end{proof}

\begin{cor}
\label{cor:HilbObstruction}
Let $R$ be a standard graded $k$-algebra,  let $h_R(z)$ be its $h$-polynomial.  Let $a\in \N$ be the order of vanishing of $h_R(z)$ at $-1$ so that  $h_{R}(z)=g(z)(1+z)^{a}$  with  $g(z)\in \Z[z]$ and $g(-1)\neq 0$. If $R$ is Koszul with  the Backelin-Roos property and  it is not a complete intersection then 

\[
g(-1)<0.
\]
\end{cor}

\begin{proof}
We may assume $k$ is infinite. Let $c$ denote the codimension of $R$ and $e(R)$ its multiplicity. Since the defining ideal   of $R$ contains a complete intersection of $c$ quadrics, one has $e(R)\leq 2^c$. Indeed by virtue of \ref{lem:MultiIneq} we have $e(R)<2^c$ since $R$ is not a complete intersection. Hence
\[
2^{c}> e(R) = h_{R}(1) = 2^{a}g(1)
\]
 It follows that $g(1)\ge 1$ and that $c>a$. 

By Proposition~\ref{prop:HilbObstruction}, the coefficients of the powers series
\[
1-\frac{g(-z)}{(1-z)^{d}}\quad\text{where $d=c-a>0$,}
\]
are non-negative. On the other hand, they are eventually given by a polynomial of degree $d-1\geq 0$ whose leading coefficient is $-g(-1)/(d-1)!$. Thus, $g(-1)< 0$ holds, as claimed.
\end{proof} 

As an application of \ref{cor:HilbObstruction} we obtain: 
\begin{cor}
\label{cor:TensorProduct}
Let $R_1$ and $R_2$ be standard graded $k$-algebras and let $T=R_1\otimes_k R_2$. Assume that $R_1$ and $R_2$ are Koszul with the Backelin-Roos property. Then we have: 
\begin{itemize} 
\item[(1)] If $R_1$ is a complete intersection then $T$ is Koszul and has the Backelin-Roos property.
\item[(2)] If both $R_1$ and $R_2$ are not complete intersection then $T$ is Koszul and  does not have  the Backelin-Roos property.
\end{itemize} 
\end{cor} 

\begin{proof} That $T$ is Koszul is a general fact, see \cite{BF}. If $R_1$ is a complete intersection then one deduces from \ref{lem:bcr-regseq} (1) that $T$ has the Backelin-Roos property. Finally, if $R_1$ and $R_2$ are both not complete intersections then, by virtue of \ref{cor:HilbObstruction},  the corresponding polynomials $g_1(z)$ and $g_2(z)$ satisfy $g_1(-1)<0$ and $g_2(-1)<0$. Now the polynomial $g(z)$ associated to $T$ is $g_1(z)g_2(z)$. 
Hence $g(-1)=g_1(-1)g_2(-1)>0$ and $T$ is not a complete intersection. It follows from \ref{cor:HilbObstruction} that $T$ does not have the Backelin-Roos property. 
\end{proof} 

The following well-known fact,  whose proof follows from the  additive formula for multiplicity \cite[Cor. 4.7.8]{BH},  has been used in the proof of \ref{cor:HilbObstruction}. 

\begin{lem}
\label{lem:MultiIneq}
Let $J, I$  be homogeneous ideals  in a polynomial ring $S$.  Assume that $J\subsetneq I$ and $\height J=\height I$.  Assume further that $J$ is pure i.e.~$\height  P=\height  J$ for every $P\in \Ass(S/J)$. Then $e(S/J)>e(S/I)$. 
\end{lem} 

The following example illustrates that the absolutely Koszul property does not descend along Golod homomorphisms. Thus, with regards to Proposition~\ref{prop:br=ak}, it is critical that the ring $Q$ in the definition of the Backelin-Roos property is a complete intersection.
 
\begin{ex}
Let $k$ be a field and, set $Q=k[x,y,z,t]/(x^2,xy,z^2)$ and let $R=Q/(zt)$. 

The ring $Q$ has the Backelin-Roos property, and hence is absolutely Koszul.  One way to verify this is to note that the ring $k[x,y]/(x^2,xy)$
is Golod (use Remark~\ref{rem:Golod2lin}), and hence has the Backelin-Roos property, and the latter is inherited by $k[x,y,z]/(x^2,xy,z^2)$; see Lemma~\ref{lem:bcr-regseq}.

The $Q$-free resolution of $R$ is given by
\[
\ldots \xrightarrow{\ z\ } Q(-4) \xrightarrow{\ z\ } Q(-3) \xrightarrow{\ z\ } Q(-2)\xrightarrow{\ zt\ } Q \to 0,
\]
so the map $Q\to R$ is Golod, by Remark \ref{rem:Golod2lin}.

Finally, it remains to recall that $R$ is not absolutely Koszul, as explained in Example~\ref{bad-Koszul}; note that $R$ is isomorphic to $k[x,y]/(x^{2},xy)\otimes_{k}k[x,y]/(x^{2},xy)$. 
\end{ex}

\section{Algebras defined by monomial relations}
\label{sect_monomial}
In this section we identify a class of Koszul algebras defined by monomial relations that have the Backelin-Roos property.  This is the content of the result below. It may be viewed as a first step towards addressing \cite[Question 4.15]{CDR} that asks for the classification of absolutely Koszul algebras defined by monomials relations.

\begin{thm}
\label{thm:ciplus2linear}
Let $L,I$ be quadratic monomial ideals in the polynomial ring $S=k[x_{1},\dots,x_{n}]$. Assume that $L$ has  a $2$-linear resolution, and $I$ is  generated by a regular sequence.  Then the ideal $(I+L)/I$ in $S/I$ also has a $2$-linear resolution.  In particular   $S/(I+L)$  has the Backelin-Roos property and hence it is absolutely Koszul.
\end{thm}

\begin{proof}
By a result of Herzog, Hibi and Zheng \cite[Theorem 3.2]{HHZ}, the ideal $L$ has linear quotients; namely,  there exists a labelling  $m_1,\ldots,m_d$ of the minimal monomial generators of $L$ such that the ideal $(m_1,\ldots,m_{i-1}):_S m_i$ is generated by variables, for each integer $i$ with $1\le i\le d$.

Set $Q=S/I$, and writing $\overline{s}$ to denote the residue class in $Q$ of an element $s\in S$, set 
\[
J_i=(\overline{m_1},\ldots,\overline{m_i})\,, \quad\text{for $1\le i\le d$.}
\]
We claim that the colon ideal $J_{i-1}:_Q \overline{m_i}$ is either the unit ideal or an ideal of $Q$ generated by residue classes of variables.  This follows by applying to the present situation the following general rule to compute colon of monomial ideals. For monomials $a_1,\dots,a_t,a$ one has 
$$(a_1,\dots,a_t):_S a=( a_i/\gcd(a_i,a) : i=1,\dots, t).$$

 We claim next that $\reg_Q J_i=2$ for all $0\le i\le d$.

Indeed, note that by the computation above, one has exact sequences of graded $Q$-modules
\[
0\to J_{i-1} \to J_i\to \frac{Q}{U}(-2) \to 0\,
\]
where the ideal $U=J_{i-1}:_Q \overline{m_i}$ is either $Q$ or   generated by residue classes of variables. 
Since $Q$ is defined by quadratic monomials, it is strongly Koszul in the sense of \cite{HHR}, so  every ideal generated by residue classes of some variables has a linear resolution; see \cite[Sect.3.3]{CDR}. Given this and the exact sequences above, an induction on $i$ yields the desired claim.

For $i=d$ one thus obtains $\reg_Q (I+L)/I=2$, as desired. It remains to note that the map $Q \to S/(I+L)$ is Golod, by Remark~\ref{rem:Golod2lin}.
\end{proof}

The example below is intended to show that in the result above, both the complete intersection $I$ and the $2$-linear ideal $J$ need to be generated by monomials. 

\begin{ex}
Let $I$ be the complete intersection of quadrics $(x_4^2-x_1x_2,x_5^2-x_2x_3)$ and $J$ the monomial ideal $(x_4x_6,x_5x_6)$, in the polynomial ring $S=\Q[x_{1},\dots,x_{6}]$.  It is easy to verify that the minimal graded $S$-free resolution of $J$ is $2$-linear. It can also be checked (by direct computation, or using some computer algebra system) that $S/(I+L)$ is not even Koszul: the minimal graded free resolution of $k$ over it has a generator in $(3,4)$. 

Similarly, the ring defined by $(xy-zt) + (x^2,y^2,z^2,t^2)$ in the ring $\Q[x,y,z,t]$ is not Koszul.
\end{ex}

\subsection*{Universally Koszul algebras}
A ring $R$ is said to be {\em universally Koszul} if every ideal generated by linear forms has a $1$-linear resolution; equivalently, the collection of ideals generated by linear forms is a Koszul filtration for $R$; see \cite{CTV}. We recall Conca's classification~\cite[Theorems 5, 6]{Con2} of universally Koszul algebras defined by monomials relations.

\begin{rem}
\label{rem:ukmonomial}
For each $m\ge 0$, let $H(m)=k[x_1,\ldots,x_m]/(x_1,\ldots,x_{m-1})^2+x_m^2$.

When the $\chara k\neq  2$, any  universally Koszul $k$-algebra defined by monomial relations is obtained from the $H(m)$'s by repeatedly performing the following operations:
\begin{itemize}
\item polynomial extension,
\item fibre product over $k$.
\end{itemize}
When $\chara k = 2$ one may need, in addition, the following operation:
\begin{itemize}
\item extending $A$ to $A[x]/(x^2)$ where $x$ is an indeterminate, of degree one.
\end{itemize}
\end{rem}

The result below settles \cite[Question 4.13(3)]{CDR} in the affirmative, for algebras defined by monomial relations.

\begin{thm}
\label{thm:uKoszul}
Universally Koszul algebras defined by monomials have the Backelin-Roos property, and hence are absolutely Koszul.
\end{thm}

\begin{proof}
Observe that the ring $H(m)$  has the Backelin-Roos property: The ring 
\[
k[x_{1},\dots,x_{m-1}]/(x_{1},\dots,x_{m-1})^{2}
\]
is Golod, as is easily verified, and hence has the Backelin-Roos property, so it remains to apply Lemma~\ref{lem:bcr-regseq}. In view of Remark~\ref{rem:ukmonomial}   the desired statement is thus an immediate consequence of Theorem~\ref{thm:br-fibre-product}  and Lemma \ref{lem:bcr-regseq}.
\end{proof}

Theorem \ref{thm:uKoszul} motivates the following

\begin{quest}
Does every universally Koszul algebra have the Backelin-Roos property, at least after extending the field? 
\end{quest}
 
\section{Veronese and Segre algebras}
\label{VeroneseSegre}
In this section we turn to Veronese subrings of polynomial rings. For certain values of the parameters involved we verify that they  have the Backelin-Roos property. On the other hand  for other values we verify that they do not have the Backelin-Roos property leaving open the question of whether they are absolutely Koszul. We prove first  the following general results. 

\begin{prop}
\label{prop:BRpoints}
Let $R$ be the homogeneous coordinate ring of $e$ points in general linear position in $\mathbb{P}_k^n$, where $n\ge 1$. When $e\le 2n$, the ring $R$ is Koszul and has the Backelin-Roos property.
\end{prop}

\begin{proof}
Kempf proved in \cite{K} that  $R$ is Koszul. In  \cite[Theorem 3.1]{CRV} it is shown that $R$ is even defined by a Gr\"obner basis of quadrics.  Indeed  the proof of \cite[Theorem 3.1]{CRV} shows  there exists an artinian reduction, say $A$, of $R$ for which there exists an element $a\in A_{1}$ with $a^{2}=0$ and $\mm^{2}=a\mm$, where $\mm=A_{\geqslant 1}$, the maximal ideal of $A$. In \cite{AIS}, such an element is called a Conca generator of $\mm$, and it follows from \cite[Theorem 1.4]{AIS} that $A$ has the Backelin-Roos property, and hence does $R$; see Lemma~\ref{lem:br-ascent}.
\end{proof}

Recall that an algebra $R$ over a field $k$ is \emph{geometrically integral} if the ring $R\otimes_k \overline{k}$  is a domain, where $\overline{k}$ is the algebraic closure of $k$.
 
\begin{thm}
\label{thm:geoint}
Let $k$ be a field of characteristic zero and $R$ a standard graded algebra that is geometrically integral. The ring $R$ is absolutely Koszul in the following cases:
\begin{enumerate}[\quad\rm(1)]
\item  $R$ is Cohen-Macaulay with $e(R)\le 2\codim R$, or
\item $R$ is a Gorenstein isolated singularity with $e(R)=2\codim R+2$ and defined by quadratic relations.
\end{enumerate}
If moreover, $k$ is algebraically closed, then $R$ also has the Backelin-Roos property.
\end{thm}

\begin{proof}
In view of Proposition~\ref{prop:br=ak}, it suffices to consider the case when  $k$ is algebraically closed, and verify that $R$ is Koszul and has the Backelin-Roos property. In particular, we may assume $R$ is a Cohen-Macaulay domain.  Also $\dim R\ge 2$, else by Hilbert's Nullstellensatz $R$ is the homogeneous coordinate ring of a point or $R\cong k$, so $\codim R=0$, contradicting our hypotheses. Furthermore, going modulo $(\dim R-2)$ general linear forms and using that by  \ref{lem:br-ascent}  the Koszulness and the Backelin-Roos property  ascend from the quotient to the original ring, we may further assume $\dim R = 2$. Observe that this process preserves the hypotheses of the desired result; for the domain property, this uses Bertini's theorem.

\bigskip

Case (1).  Let $x$ be a general linear form in $R$. By Harris' general position theorem \cite[page 109]{ACGH},  the ring $R/(x)$ is the homogeneous coordinate ring of $e(R)$ points in general linear position. Since $e(R)\le 2\codim R$,  Proposition \ref{prop:BRpoints} yields that $R/(x)$, and hence also $R$, is Koszul and has the Backelin-Roos property.

\bigskip

Case (2). Since $R$ has dimension two and is an Gorenstein isolated singularity, it is normal. By the discussion at the end of page 7 in \cite{Eis}, it then follows that  $R$ is the homogeneous coordinate ring of a canonical curve. As $R$ is quadratic, by the discussion before Theorem 5.3 in \cite{CRV} and Theorem \ref{thm:glind-flatdim}, we can further replace $R$ by the homogeneous coordinate ring of a set $X$ of $2n+2$ distinct points in $\PC^n$ with the following properties:
\begin{enumerate}
\item $X=Y_1\cup Y_2$ where $Y_1$, $Y_2$ are sets of $n+1$ distinct points such that the points in $Y_i$ belong to a hyperplane $H_i$ in $\PC^n$ for $i=1,2$, and,
\item $H_1$ contains no point of $Y_2$ and $H_2$ contains no point of $Y_1$.
\end{enumerate}
In particular, $X$ and $R$ satisfy the condition of \cite[Theorem 5.3]{CRV}. From the last paragraph in the proof of \emph{ibid.}, there exists an $R$-regular linear form $l$ such that in the ring $A=R/lR$ there exists a non-trivial linear form $a$ with $a^2=0$ and $\rk (A_1\xrightarrow{\ a\ } A_2)=n-1$. By the argument in the proof  of \cite[Proposition 2.13(a)]{CRV}, one  deduces that $(0\colon a)=(a)$. Then, since $A$ is Gorenstein, so is the ring $A/(a)$, with Hilbert series $1+(n-1)z + z^{2}$. Since $k$ is algebraically closed, it follows from \cite[Lemma 3]{Con0} that $A/(a)$ has a Conca generator, in the sense recalled in the proof of Proposition~\ref{prop:BRpoints}. It now remains to apply \cite[Proposition 3.9]{HS} to deduce that $A$,
and so also $R$, has the Backelin-Roos property.
\end{proof}

\begin{rem}
Concerning the preceding result, and Proposition~\ref{prop:HilbObstruction}, if $h(z)=1+az+bz^2$ for positive integers $a,b$ the coefficients of power series expansion of the rational function $1-h(-z)/(1-z)^a$ are all non-negative if and only if $a>b$. This is exactly the condition $h(1)\leq 2 a$, that is, $e(R)\leq 2\codim(R)$. 
\end{rem}

The following is a direct corollary of Theorem \ref{thm:geoint}.

\begin{cor}
\label{cor:Veronese}
Assume $\chara k=0$. Let $S=k[x_1,\ldots,x_n]$ be a polynomial ring in $n$ variables and $c$ an integer $\ge 2$.
The Veronese subring $S^{(c)}$ of $S$ is absolutely Koszul in the following cases:
\begin{enumerate}[\quad\rm(1)]
\item $n\le 3$ and any $c$;
\item $n\le 4$ and $c\le 4$;
\item $n\le 6$ and $c=2$.
\end{enumerate}
When, in addition, $k$ is algebraically closed, $S^{(c)}$ also has the Backelin-Roos property.
\end{cor}

\begin{proof}
As is well-known, $S^{(c)}$ is geometrically integral, Cohen-Macaulay, and satisfies
\[
 e(S^{(c)})=c^{n-1} \quad\text{and}\quad \codim S^{(c)}=\binom{n+c-1}{c}-n\,.
 \]
A direct verification shows that when $n\le 3$ and $c\ge 2$, or when $n=4$ and $c=3$, or when $n=5$ and $c=2$, one has $e(S^{(c)})\le 2\codim S^{(c)}$ holds; in fact these are the only values for which the inequality holds. In these cases, desired statement follows from Theorem \ref{thm:geoint}(1).

It remains to tackle the case $n=4=c$, and $n=6$ and $c=2$; set $R=S^{(c)}$. In the first case $e(R)=32$ and $\codim R=15$, whilst in the second $e(R)=64$ and $\codim R=31$. In either case $R$ is a Gorenstein, quadratic with isolated singularity and with $e(R)=2\codim(R)+2$, so Theorem \ref{thm:geoint}(2) applies and yields the desired conclusion.
\end{proof}

Using a Gr\"obner basis argument, we can show that second Veronese subrings of polynomial rings in $n\le 6$ variables satisfy the Backelin-Roos property for any field $k$.  The key observation is the following: 

\begin{lem}
\label{lem:GBdef}
Let $I$ be an ideal of a polynomial ring $S$ and $<$ a term order. If $\ini(I)=U+V$ where $U$ and $V$ are monomial quadratic ideals such that $U$ is a complete intersection and $V$ has a $2$-linear resolution, then $S/I$ is Koszul and has the Backelin-Roos property.
\end{lem} 

\begin{proof}  
By assumption $U=(u_1,\dots,u_p)$ where $u_i$ is a quadratic monomials and $\gcd(u_i,u_j)=1$ whenever $i\neq j$. Furthermore $V=(v_1,\dots,v_t)$ where the $v_i$'s are quadratic monomials. Consider a  Gr\"obner basis of  $I$, say  $g_1,\dots,g_p,f_1,\dots,f_t$, such  that $\ini(g_i)=u_i$ and $\ini(f_j)=v_j$ for every $i$ and $j$.  Let $J=(g_1,\dots,g_p)$.  It follows immediately that $g_1,\dots,g_p$ is a regular sequence and that $\ini(J)=U$.  Since $\ini(I)=U+V$ we may apply Theorem~\ref{thm:ciplus2linear} and deduce that $\reg_{S/\ini(J)}(S/\ini(I))=1$. Then the standard Gr\"obner deformation \cite[Prop.3.13]{BC} argument yields that  
\[
\reg_{S/J}(S/I)\leq \reg_{S/\ini(J)}(S/\ini(I))=1\,.
\]
It follows that the map $S/J\to S/I$ is Golod. 
\end{proof} 

\begin{cor}
\label{cor:anyk}
Let $k$ be an arbitrary field and $S=k[x_1,\dots,x_n]$. Then $S^{(2)}$ has the Backelin-Roos property for every  $n\leq 6$. 
\end{cor} 

\begin{proof} 
Let $T=k[x_{ij}:i\le j]=\Sym(S_2)$ and $\phi\colon T\to S^{(2)}$ be the canonical morphism mapping the variable $x_{ij}$ to the monomial $x_ix_j$ for $1\le i\le j\le n$.  We will  apply Lemma~\ref{lem:GBdef}  to the $\ker \phi$. We do it for $n=6$. For smaller values of $n$ the argument is easier. 
  The kernel of $\phi$ contains all the binomials of the from  $x_{ab}x_{cd}-x_{ef}x_{gh}$
where $\{a,b,c,d\}=\{e,f,g,h\}$ as  multisets. Indeed these binomials form Gr\"obner basis of $\ker \phi$ with respect to every degree reverse lexicographic term order on $T$. 
Let $<$ be the degree reverse lexicographic term order on $T$ associated to the following total order of the variables
\[
x_{11}<x_{22}<\dots<x_{66}<x_{12}<x_{34}<x_{56}< \dots \mbox{ all the other variables in any order}.
\]
For every $1\leq i<j\leq 6$ consider the the binomial
\[
g_{ij}=x_{ij}^2-x_{ii}x_{jj}
\]
of $\ker \phi$ whose leading term is $x_{ij}^2$.  
By what we have said above and taking into consideration the underlying  multigraded graded structure  we have that 
\[
\ini(\ker \phi)=( x_{ij}^2 : 1\leq i<j\leq 6)+L
\]
where $L$ is generated   by square-free monomials of degree $2$ in the variables $x_{ij}$ with $i<j$.  We have to show that $L$ has a $2$-linear resolution over the polynomial ring $T$. In order to do that we associate to $L$ a graph $G$ with vertices $x_{ij}$ with $1\leq i<j\leq 6$ and edges connecting $x_{ij}$ to $x_{hl}$ if and only if $x_{ij}x_{hl}\not\in L$. According to Fr\"oberg theorem \cite{F}, it is enough to show that $G$ is chordal, i.e.~every cycle of length $\ge 4$ has a chord. A simple  count shows that $G$ has $15$ vertices and $15$ edges. 
 It is also clear that the following are edges of $G$:  
\[
\begin{array}{l}
 \text{ $6$ edges of the form }   x_{12}-x_{ab}, ~ \text{with $\{a,b\}$ disjoint from $\{1,2\}$,}\\
 \text{ $5$ edges of the form } x_{34}-x_{ab},  ~ \text{with $\{a,b\}$ disjoint from $\{3,4\}$ and ${(a,b)}\neq {(1,2)}$,}\\
\text{ $4$ edges of the form }  x_{56}-x_{ab},  ~ \text{with $\{a,b\}$ disjoint from $ \{5,6\}$ and ${(a,b)}\neq {(1,2)}$  and ${(3,4)}$  }.
\end{array}
\]
Since $6+5+4=15$, these are all the edges of $G$. We can now see that the graph $G$ is chordal, since every vertex which is different from
$x_{12}, x_{34}$ and $x_{56}$ has only degree $1$. For example $x_{35}$ is connected to $x_{12}$ but not to $x_{34}$ or to $x_{56}$ and so on. This finishes the proof of the statement.
 \end{proof} 

However most Veronese subrings  do not have  the Backelin-Roos property.

\begin{lem}
\label{lem:obstructedVeronese}  
Set $S=k[x_1,\dots,x_n]$ and let $h_{n,c}(z)$ denote the $h$-polynomial of $S^{(c)}$.  One has an inequality $h_{n,c}(-1)>0$ for the following values $(n,c)=(7,2), (5,4), (5,3), (4,c)$ with $c>4$. In particular $S^{(c)}  $ does not have the Backelin-Roos property from those values on $(n,c)$. 
\end{lem} 

\begin{proof} 
Set $R=S^{(c)}$. The Hilbert function of $R$ is given by 
\[
\HF(R,i)=\binom{n-1+ic}{n-1}\,.
\]
It coincides with the Hilbert polynomial for all $i\geq 0$. This implies that $\deg h_{n,c}(z)<n$, say 
\[
h_{n,c}(z)=\sum_{i=0}^{n-1} h_i z^i
\]  
where 
\[
h_i=\sum_{j\geq 0} (-1)^{i-j} \binom{n-1+jc}{n-1} \binom{n}{i-j}.
\]
When $(n,c)$ equals $(7,2)$, $(5,4)$, and $(5,3)$ this gives
\begin{align*}
h_{7,2}(z) &=1+21z+35z^2+7z^3 \\
h_{5,4}(z) &=1+65z+155z^2+35z^3\\
h_{5,3}(z) &=1+30z+45z^2+5z^3 
\end{align*}
These functions are  positive when evaluated at $z=-1$.  Finally for $(n,c)=(4,c)$ one uses the formula above to compute $h_0,h_1,h_2,h_3$ and then obtains: 
\[
h_{4,c}(-1)=1/3(c-4)(c^2 + 4c - 6)
\]
that  is clearly positive when $c>4$. 
\end{proof} 

\begin{rem} 
\label{rem:whythis}
The set $T=\{ (7,2),  (5,4), (5,3), (4,c) : c>4\}$ of pairs  we have tested in \ref{lem:obstructedVeronese} is chosen so that for any pair $(n,c)$ that does not satisfy the conditions of  Corollary~\ref{cor:Veronese} there exists a pair   $(m,c)$ in $T$ with $m\leq n$, so that  $k[x_1,\dots,x_m]^{(c)}$ is an algebra retract of $k[x_1,\dots,x_n]^{(c)}$. Therefore a positive answer to Question \ref{que:retract} in combination with Lemma~\ref{lem:obstructedVeronese} would imply that the list of  Veronese algebras with the Backelin-Roos property given in  Corollary~\ref{cor:Veronese} is complete.

\medskip

A caveat:  $h_{n,c}(-1)\leq 0$ for certain values of $(n,c)$. Indeed, a direct computation yields
\[
h_{6,7}(z)=6z^5 + 1251z^4 + 7872z^3 + 6891z^2 + 786z + 1
\]
so $h_{6,7}(-1)=-521$. Nevertheless,  
\[
1-\frac{h_{6,7}(-z)}{(1-z)^{786}}= 301614z^2 +156453836z^3+\cdots +\alpha z^{121} +\cdots 
\]
where $\alpha\simeq -(1.5)  10^{152}$.  Hence, by virtue of Proposition~\ref{prop:HilbObstruction},  the corresponding Veronese subring  does not have the Backelin-Roos property.   

\medskip

Indeed, computational evidence suggests that the only Veronese algebras that satisfy the conditions of Proposition~\ref{prop:HilbObstruction} are  those listed in Corollary~\ref{cor:Veronese}. 
\end{rem}

\subsection*{Segre products} 
Let $S_{m,n}$ denote the Segre product of polynomial rings $A=k[x_1,\ldots,x_m]$ and $B=k[y_1,\ldots,y_n]$ where $m\le n$, and let $h_{m,n}$ denote its $h$-polynomial. Obviously the Hilbert function of $S_{m,n}$ is 
\[
\HF(S_{m,n},i)= \dim A_i\dim B_i=\binom{m-1+i}{m-1}\binom{n-1+i}{n-1}\,.
\]
One deduces immediately that 
 \begin{gather*}
e(S_{m,n})=\binom{m+n-2}{m-1}\quad\text{and} \\
h_{m,n}(z)=\sum_{i=0}^{m-1} \binom{m-1}{i}\binom{n-1}{i}z^i.
\end{gather*}
In analogy with the Veronese case we have: 

\begin{prop}
Assume $k$ has characteristic $0$. The ring $S_{m,n}$ is absolutely Koszul when 
\begin{enumerate}[\quad\rm(1)]
\item $m\leq 2$,
\item $m=3$ and $n\le 5$,
\item $m=n=4$.
\end{enumerate}
When, in addition, $k$ is algebraically closed, these $S_{m,n}$ also have the Backelin-Roos property.
\end{prop}
 
The Segre products $S_{3,6}, S_{4,5}$  do not have the Backelin-Roos property, as can be verified by computations similar to the one for Lemma~\ref{lem:obstructedVeronese}. And the argument of Remark~\ref{rem:whythis} applies to this situation too. 

\section{Global linearity defects}
\label{sect_glind}
In this section we investigate bounds on the linearity defect of modules. We begin with the following observation.

\begin{lem}
\label{lem:deg2}
Let $R$ be a standard graded $k$-algebra and $N$ a finitely generated graded $R$-module. If $x\in R$ is an $N$-regular element of degree $d\ge 2$, then there is an equality
\[
\lind_R(N/xN)=\lind_R N+1\,.
\]
\end{lem}

\begin{proof}
Consider the exact sequence of $R$-modules
\[
0\to N(-d) \xrightarrow{\ x\ } N \to N/xN \to 0.
\]
Let $F$ be the minimal graded free resolution of $N$ over $R$. The mapping cone, say $G$, of the morphism $\phi\colon F(-d) \xrightarrow{\ x\ } F$ of complexes is then the minimal graded $R$-free resolution of $N/xN$. Since $\Img(\phi) \subseteq \mm^2 F$, where $\mm$ is the homogeneous maximal ideal of $R$, it is easy to see that 
\[
\linp^R(G)=\linp^R F \oplus (\linp^R F)[-1]\,.
\]
The desired equality is now immediate. 
\end{proof}

\begin{thm}
Let $R$ be a standard graded $k$-algebra. Then $\glind R[x]=\glind R+1.$
\end{thm}

\begin{proof}
To begin there, there are inequalities
\[
\glind R \le \glind R[x] \le \glind R+1
\]
where the one on the left is from Proposition~\ref{prop:retracts}, applied to the canonical surjection $R[x]\to R$, and the one on the right is from Theorem~\ref{thm:glind-flatdim}, applied to the inclusion $R\to R[x]$. We may thus assume that $\glind R$ is finite and then it suffices to check that $\glind R[x] \ge \glind R+1$.

Let $M$ be a finitely generated graded $R$-module with $\lind_R M=\glind R$. Setting $S=R[x]$ and $N=S\otimes_{R}M$, there are equalities
\[
\lind_{S}(N/x^2N) =\lind_S N +1 = \lind_{R}M + 1
\]
where the one on the left is by Lemma \ref{lem:deg2} and the one on the right holds because $R\to S$ is flat. This justifies the desired inequality.
\end{proof}

\begin{prop}
\label{prop:glind-inequality}
Let $R$ be a standard graded $k$-algebra, $M$ a finitely generated graded $R$-module. Then there is an inequality
\[
\lind_R M+\depth_{R} M \le \glind R.
\]
In particular, $\glind R \ge \depth R$.
\end{prop}

\begin{proof}
We proceed by induction on $\depth M$. The base case $\depth_{R} M=0$ is trivial. When $M$ has positive depth, choose an $M$-regular element $x\in R$ of degree $\ge 2$. Then Lemma \ref{lem:deg2} gives the first equality below
\begin{align*}
\lind_R M+\depth M 
&= \lind_R(M/xM) - 1 + \depth_{R}(M/xM) + 1 \\
&= \lind_R (M/xM)  + \depth_{R}(M/xM)  \\
& \le \glind R
\end{align*}
and the last one is the induction hypothesis. 
\end{proof}

\begin{cor}
Let $R$ be a Cohen-Macaulay standard graded $k$-algebra satisfying $\reg R=1$. Then $R$ has minimal multiplicity and $\glind R=\dim R$.

In particular, if $f$ is a non-zero quadratic form in $k[x_1,\ldots,x_n]$, then $\glind (k[\xb]/(f)) = n-1.$
\end{cor}

\begin{proof}
We may assume $k$ is infinite; see Lemma~\ref{lem:field-extension}. Given Proposition~\ref{prop:glind-inequality}, it remains to show $\glind R \le \dim R$. This was claimed in \cite[Remark, p.~ 21]{HIy}; we include a proof for the benefit of the second author who could not recollect the details of the argument. Note that $\glind R \le \glind (R/Rx) +1$ if $x\in R_1$ is $R$-regular; this is by Theorem~\ref{thm:glind-flatdim}. We may thus reduce to the case when $\dim R=0$. Note that the regularity of $R$ and its multiplicity remain unchanged.

Let $R=P/I$ where $P$ is a polynomial ring and $I\subseteq \mm^2$, where $\mm=P_{\geqslant 1}$.  Since $I$ has $2$-linear resolution and $\projdim_PR=n$, there is an equality of Hilbert series
\[
H_R(z)(1-z)^n=1-\beta_1z^2+\cdots+(-1)^{n}\beta_nz^{n+1},
\]
where $\beta_i\neq 0$ is the $i$th Betti number of $R$ over $P$. Therefore by comparing degrees of the polynomials, $R_i=0$ for $i\ge 2$, so $I=\mm^2$. Then every $R$-module is Koszul, so $\glind R=0$.

The last statement holds as $k[\xb]/(f)$ is Cohen-Macaulay of dimension $n-1$.
\end{proof}

\begin{thm}
If $R$ is defined by monomial relations, then $\glind R \ge \dim R$.
\end{thm}

\begin{proof}
Suppose $R=P/I$ where $P=k[x_1,\ldots,x_n]$ is a polynomial ring and $I$ is a monomial ideal; we may assume it is quadratic, for else $\lind_{R}k$ is infinite. Reordering the variables if necessary we may assume that in the primary decomposition of $I$ the component of minimal height is $(x_1^2,\ldots,x_q^2,x_{q+1},\ldots,x_s)$, where $s=n-\dim R$.

We claim that $\lind_R (R/J) = \dim R$ where $J=(x_{1},\dots,x_{s},x_{s+1}^2,\ldots,x_n^2)$. 

Indeed, set $S=k[x_{s+1},\ldots,x_n]$ and let $R\to S$ be the canonical surjection. Note that the composition of the inclusion $S\to R$ with the map $R\to S$ is the identity on $S$. Moreover $\lind_R S=0$, since $R$ is strongly Koszul. Therefore, noting that the action of $R$ on $R/J$ factors through $S$, from Proposition \ref{prop:retracts} one gets the first equality below:
\[
\lind_R (R/J) =\lind_S(R/J) = \lind_S (S/(x_{s+1}^2,\ldots,x_n^2))=n-s\,.
\]
The last equality is a direct computation; one can get it from Lemma~\ref{lem:deg2}.
\end{proof}

Motivated by the last result, we ask

\begin{quest}
Does the inequality $\glind R \ge \dim R$ always hold?
\end{quest}

By Proposition \ref{prop:glind-inequality} this is the case for Cohen-Macaulay rings; more generally, it holds when $R$ has a maximal Cohen-Macaulay module, and in particular when $\dim R\le 2$.

\end{document}